\documentclass[twoside,12pt]{article}
\usepackage{geometry}
\usepackage{amsmath,amssymb,amsfonts,amsthm}
\usepackage{txfonts}
\usepackage{todonotes}
\usepackage{tablists}
\usepackage{graphicx}
\usepackage{mathrsfs}
\usepackage{color}
\usepackage{float}
\usepackage{extarrows}
\usepackage{exscale}
\usepackage{titlesec}
\titleformat{\section}[block]{\large\center\sc}{\arabic{section}}{0.5em}{}[] 
\usepackage{relsize}
\usepackage{color}
\definecolor{teal}{RGB}{0,128,172}
\definecolor{pur}{RGB}{224,104,255}
\usepackage{hyperref}
\hypersetup{hypertex=true,
            colorlinks=true,
            linkcolor=teal,
            anchorcolor=blue,
            citecolor=pur}
\usepackage[titles]{tocloft}
\usepackage{fancyhdr}
\theoremstyle{plain}
\newtheorem{theorem}{Theorem}[section]
\newtheorem{lemma}[theorem]{Lemma}
\newtheorem{corollary}[theorem]{Corollary}
\newtheorem{proposition}[theorem]{Proposition}

\newtheorem{remark}[theorem]{Remark}

\let\oldsection\section
\renewcommand\section{\setcounter{equation}{0}\oldsection}

\def\be{\begin{equation}}
\def\ee{\end{equation}}
\def\bes{\begin{equation*}}
\def\ees{\end{equation*}}
\def\bs{\begin{split}}
\def\es{\end{split}}
\def\bali{\begin{aligned}}
\def\eali{\end{aligned}}

\def\bR{{\mathbb R}}
\def\un{\underbrace}
\def\al{\alpha}

\def\th{\theta}

\def\Dl{\Delta}
\def\lt{\left}
\def\rt{\right}

\def\i{\infty}
\def\p{\partial}
\def\f{\frac}
\def\na{\nabla}

\def\q{\quad}

\def\bl{\boldsymbol}

\def\mR{\mathbb{R}}

\def\mH{\mathcal{H}}

\allowdisplaybreaks
\def\mD{\mathcal{D}}

\def\cd{\cdot}
\def\les{\lesssim}

\def\mf{\mathfrak}

\begin{document}
\title{\bf\Large  On local well-posedness of 3D ideal Hall-MHD system with an azimuthal magnetic field}

\author{\normalsize\sc Zijin Li}

\date{}

\maketitle

\begin{abstract}
In this paper, we study the local well-posedness of classical solutions to the ideal Hall-MHD equations whose magnetic field is supposed to be azimuthal in the $L^2$-based Sobolev spaces. By introducing a good unknown coupling with the original unknowns, we overcome difficulties arising from the lack of magnetic resistance, and establish a self-closed $H^m$ with $(3\leq m\in\mathbb{N})$ local energy estimate of the system. Here, a key cancellation related to $\th$ derivatives is discovered. In order to apply this cancellation, part of the high-order energy estimates is performed in the cylindrical coordinate system, even though our solution is not assumed to be axially symmetric.

During the proof, high-order derivative tensors of unknowns in the cylindrical coordinates system are carefully calculated, which would be useful in further researches on related topics.
\medskip



\end{abstract}

\tableofcontents

\section{Introduction}
We consider the 3D Hall-MHD system
\be\label{Hall}
\left\{\begin{array}{l}
\partial_t \bl{v}+\bl{v} \cdot \nabla \bl{v}+\nabla P-\mu\Dl\bl{v}=\bl{h} \cdot \nabla \bl{h}, \\
\partial_t \bl{h}+\bl{v} \cdot \nabla \bl{h}+\nabla \times((\nabla \times \bl{h}) \times \bl{h})-\nu\Dl\bl{h}=\bl{h} \cdot \nabla \bl{v}, \\
\nabla \cdot \bl{v}=0,\\
\nabla \cdot \bl{h}=0, \\
\end{array}\right.\q\text{for}\q (t,x)\in\mathbb{R}_+\times\mathbb{R}^3\,,
\ee
with initial data
\be\label{INIT}
\bl{v}(0,x)=\bl{v}_0(x),\q\text{and}\q\bl{h}(0,x)=\bl{h}_0(x)\q\q\text{for}\q x\in\mathbb{R}^3.
\ee
Here, divergence-free three dimensional vector fields $\bl{v}=(v_1, v_2, v_3)$ and $\bl{h}=(h_1, h_2, h_3)$ represent the velocity field and magnetic field, respectively. $P\in\bR$ represents the scalar pressure. $\mu\geq0$ is the fluid viscosity, while $\nu\geq0$ stands for the magnetic resistance.

When $\mu=\nu=0$, system \eqref{Hall} is called the ideal Hall-MHD system. By the influence of the Hall-effect (which is represented as the 2nd-order nonlinear structure $\nabla \times((\nabla \times \bl{h}) \times \bl{h})$ in the magnetic equation), together with the lack of compensating by the resistivity $\Dl\bl{h}$, the loss of derivative seems inevitable. Recently, various ill-posedness results of Hall- and electron-MHD systems were shown in Jeong-Oh \cite{Jeong2021APDE}.

In this paper, we consider the ideal Hall-MHD system whose magnetic field is azimuthal, that is, $\bl{h}$ has the following form of expression:
\be\label{Azi}
\bl{h}=h_\theta(t, r, z,\theta) \bl{e_\theta}\,.
\ee
Here $h_\theta$ is a scalar function, while $\bl{e_\theta}$ is the unit vector of the horizontal swirl direction in the following cylindrical basis:
\[
\bl{e_r}=\big(\frac{x_1}{r}, \frac{x_2}{r}, 0\big), \quad \bl{e_\theta}=\big(-\frac{x_2}{r}, \frac{x_1}{r}, 0\big), \quad \bl{e_z}=(0,0,1),
\]
and
$$
r=\sqrt{x_1^2+x_2^2}, \quad \theta=\arctan \frac{x_2}{x_1}, \quad z=x_3\,.
$$

Since $\bl{h}$ is supposed to be azimuthal, its divergence field can be simply written:
\[
\na\cd\bl{h}=\f{1}{r}\p_\theta h_\theta\,.
\]
Thus the divergence free property of the magnetic field implies $h_\theta$ is independent with $\theta$. Rewriting \eqref{Hall} with $\mu=\nu=0$ in the aforementioned cylindrical coordinates, one derives
\be\label{(2.1)}
\left\{\begin{split}
&\p_tv_r+\Big(v_r \partial_r+\frac{1}{r} v_\theta \partial_\theta+v_z \partial_z\Big) v_r-\frac{v_\theta^2}{r}+\frac{2}{r^2} \partial_\theta v_\theta+\partial_r P=-\frac{h_\theta^2}{r}, \\[1mm]
&\p_tv_\theta+\Big(v_r \partial_r+\frac{1}{r} v_\theta \partial_\theta+v_z \partial_z\Big) v_\theta+\frac{v_\theta v_r}{r}-\frac{2}{r^2} \partial_\theta v_r+\frac{1}{r} \partial_\theta P=0, \\[1mm]
&\p_tv_z+\Big(v_r \partial_r+\frac{1}{r} v_\theta \partial_\theta+v_z \partial_z\Big) v_z+\partial_z P=0, \\[1mm]
&\p_th_\theta+\left(v_r \partial_r+v_z \partial_z\right) h_\theta-\frac{1}{r} h_\theta \partial_\theta v_\theta-\frac{h_\theta v_r}{r}=\f{\p_zh_\theta^2}{r}, \\[1mm]
&\nabla \cdot \bl{v}=\partial_r v_r+\frac{v_r}{r}+\frac{1}{r} \partial_\theta v_\theta+\partial_z v_z=0.\\[1mm]
\end{split}\right.
\ee
Where
\[
\bl{v}=v_r(t,r,z,\theta)\bl{e_r}+v_\theta(t,r,z,\theta)\bl{e_\theta}+v_z(t,r,z,\theta)\bl{e_z}\,.
\]

Various physical phenomena and examples are related to conductive fluid flows with azimuthal magnetic fields. Physicist often uses a powerful azimuthal magnetic field to confine plasma in the shape of a torus, such as the famous \emph{tokamak device} that producing controlled thermonuclear fusion power. Meanwhile, the magnetic field intuited by a current $I$ in the straight electric wire is azimuthal, which is given by
\[
\bl{h}=\f{\mu_0I}{2\pi r}\bl{e_\theta}\,.
\]
Here $\mu_0$ is the permeability of free space.
\begin{figure}
\centering
\includegraphics[scale=0.3]{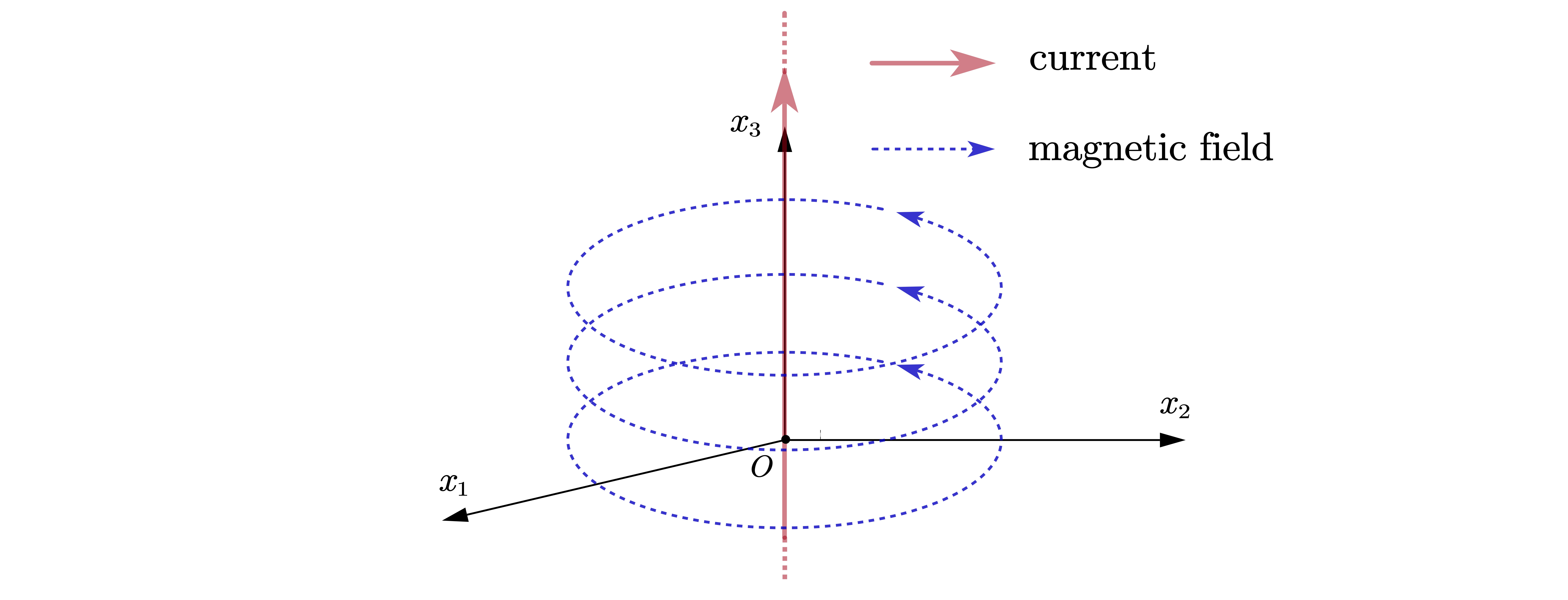}
\caption{The azimuthal magnetic field intuited by a straight current $I$}
\end{figure}

Before stating the main theorem of the present paper, let us denote a \emph{good unknown} related to the magnetic field: $\mathcal{H}:=\f{h_\theta}{r}$. It will play important role during the proof. From \eqref{(2.1)}$_4$, one deduces
\be\label{EHH}
\p_t\mH+(v_r\p_r+v_z\p_z)\mH-\mH\f{\p_\theta v_\theta}{r}-2\mH\p_z\mH=0\,.
\ee
Denoting $\bl{b}:=v_r\bl{e_r}+v_z\bl{e_z}$, one can rewrite \eqref{EHH} as
\[
\p_t\mH+\bl{b}\cdot\na\mH-\mH\f{\p_\theta v_\theta}{r}-2\mH\p_z\mH=0\,.
\]
Noticing that $\mH$ is independent with $\theta$, the vector field $\bl{b}$ in the above equation could also be replaced by $\bl{v}$.

Here goes the main result of the paper: the local well-posedness of the system \eqref{(2.1)} in $L^2$-based the Sobolev space $H^m(\bR^3)$.

\begin{theorem}
Assume $(\bl{v}_0,\bl{h}_0,\mH_0)\in H^m(\bR^3)$ with $3\leq m\in\mathbb{N}$, and $\bl{h}$ is azimuthal as shown in \eqref{Azi}, then there exists $T_*>0$, depending only on $\|(\bl{v}_0,\bl{h}_0,\mH_0)\|_{H^m}$, such that the system \eqref{Hall}--\eqref{INIT} has a unique strong solution $(\bl{v},\bl{h},\mH)$ on $[0,T_*]\times\mathbb{R}^3$, and it satisfies
\[
(\bl{v},\bl{h},\mH)\in L^\i(0,T_*,H^m(\bR^3))\,.
\]
\end{theorem}

\qed

The 3D Hall-MHD equations have been mathematically investigated in several works. Existence of global weak solutions was given in \cite{ADFL2011KRM} with $\mu=\nu=1$, and later \cite{CS2013JDE} showed the temporal decay estimate of global weak or strong solutions. Global weak solutions of \eqref{Hall} with both fluid viscosity $(\mu>0)$ and magnetic resistivity $(\nu>0)$, also local classical solutions of \eqref{Hall} with only magnetic resistivity were obtained in \cite{CDL2014AIHPL}. A blow up criterion and small data global existence for classical solutions ($\mu\geq0$ and $\nu>0$) were also given in \cite{CDL2014AIHPL}, and it was later sharpened in \cite{CL2014JDE}. When the initial magnetic field is close to a background magnetic field that satisfying a Diophantine condition, a global well-posedness for small solutions to the inviscid ($\mu=0$ and $\nu>0$) Hall-MHD system on $\mathbb{T}^3$ was given in \cite{Zhai2021JMFM}.

The Hall effect, which is described by the nonlinear term with second order derivatives $\na\times((\na\times\bl{h})\times\bl{h})$, creates much more trouble in deriving the well-posedness of the Hall-MHD system even locally. Without the help of the magnetic resistivity, controlling the Hall effect term by the same order as the energy functional and then build a self-closed energy estimate in Sobolev spaces seem impossible without some breaking through idea. Huge efforts have been made in recent years and many partial results were given. If the magnetic diffusion given by a fractional Laplacian operator $(-\Dl)^{\al}$ with $\al>\f{1}{2}$, Chae-Wan-Wu \cite{CWW2015JMFM} showed the local well-posedness. Jeong-Kim-Lee \cite{Jeong2018AMP} showed the local well-posedness and blow-up for the axially symmetric inviscid Hall-MHD system. On the other hand, Chae-Weng \cite{CW2016AIHPL} showed the non-resistive Hall-MHD system is not globally in time well-posed in any Sobolev space $H^s(\bR^3)$ with $s>\f{7}{2}$. Jeong-Oh \cite{Jeong2021APDE} proved various ill-posedness results for the Cauchy problem of the Hall- and electron-MHD system without resistivity, and they also claimed several well-posedness results if the initial magnetic field satisfies certain geometric conditions.

In fact, even for a special case the $\bl{v}\equiv0$, equation \eqref{EHH} degenerates to a one dimensional Burgers' equation:
\[
\p_t\mH-\p_z\mH^2=0\,,
\]
which will certainly generate finite time blow up even for $C^\i_0(\bR^3)$ initial data. In the situation of Hall-MHD systems, these solutions are known as the \emph{KMC waves}, see Kingsep-Mokhov-Chukbar \cite{KMC1984}, also \cite{CG1995}. Apparently, these shock wave solutions only exists when the Hall-effect works.

Nevertheless, if the magnetic field is supposed to be azimuthal, the Hall-effect term $\na\times((\na\times\bl{h})\times\bl{h})$ is simplified to
\[
-\f{\p_zh_\theta^2\bl{e_\theta}}{r},
\]
and the good unknown $\mH=\f{h_\theta}{r}$ is introduced. In this way, we rewrite the magnetic equation as
\[
\partial_t \bl{h}+\bl{v} \cdot \nabla \bl{h}-2\mH\p_z\bl{h}=\bl{h} \cdot \nabla \bl{v}.
\]
This observation motivates us to regard $\mH$ as an novel unknown quantity, running in parallel with $\bl{v}$ and $\bl{h}$, although $\mathcal{H}$ exhibits a one higher order derivative compared to $\bl{h}$ in the sense of scaling. From the perspective of the author, this constitutes a pivotal insight enabling the control of the higher-order Hall-effect term autonomously, without recourse to magnetic resistance assistance.

Recently, with the help of this good unknown $\mH$, Li-Yang \cite{Li-Yang2022} obtained a blow up criterion, which was imposed only on $\p_z\mH$, for the axially symmetric no-swirl non-resistive Hall-MHD system. See also Lei \cite{Lei2015} where the author showed the global well-posendess of strong large solutions to axially symmetric MHD system, by applying the conservation property of $\mH$.

However, if the axial symmetry is discarded, a serious trouble arises due to the appearance of $\f{\mH\p_\theta v_\theta}{r}$ in \eqref{EHH}. More precisely, when carrying out the $m$-order energy estimate, the integral
\[
\int_{\mR^3}\na^m\big(\mH\f{\p_\theta v_\theta}{r}\big)\na^m\mH dx
\]
contains a irresistible $(m+1)$-order derivative of unknowns. Without the help of high-order viscous terms, the loss of one-order derivative seems inevitable, and one cannot close the energy estimate in a functional space that only allows finite order derivatives.

On the other hand, one notices the following fact: Given two functions $f,g\in C_c^\i$, with $g$ is independent with $\theta$, then
\be\label{ENS}
\begin{split}
\int_{\mR^3}\p_\theta fgdx&=\int_{-\i}^\i\int_0^\i\int_0^{2\pi}\f{d}{d\theta}\big[f(r,z,\theta)g(r,z)\big]d\theta rdrdz\\
&=\int_{-\i}^\i\int_0^\i\big[g(r,z)f(r,z,\theta)\Big|_{\theta=0}^{2\pi}\big] rdrdz=0\,.
\end{split}
\ee
Since $\mH$ is independent with $\theta$, it seems hopeful to eliminating this trouble term in this way, and this do work if the $m$th order derivatives are all taken in $x_3-$direction. Unfortunately, this will not work if $\na^m$ consists $x_i-$derivatives, with $i=1,2$. More precisely, due to the following nontrivial commutators:
\[
\begin{split}
\big[\p_{x_1},\f{\p_\theta}{r}\big]&=\sin\theta\f{\p_r}{r}\neq0\,;\\
\big[\p_{x_2},\f{\p_\theta}{r}\big]&=-\cos\theta\f{\p_r}{r}\neq0\,,
\end{split}
\]
it is impossible to write $\na^m\f{\p_\theta v_\theta}{r}$ with the form $\p_\theta F$, where $F$ is a function depending on $v_\theta$. What is more, noticing that
\[
\p_{x_1}\mH=\cos\theta\p_r\mH,
\]
it is clear that in spite that $\mH$ is independent with $\theta$, $\na^m\mH$ is not. Heuristically, to apply the cancelation property \eqref{ENS}, one has to write derivatives of $\mH$ in the cylindrical coordinates. By the Sobolev imbedding, one needs to calculate at least third-order derivatives of $\mH$, together with $\bl{v}$, in the cylindrical coordinates.

However, when performing the $H^m$ $(m\geq3)$ energy estimates of $\mH$ in the cylindrical coordinate system, difficulties arise due to the lack of axial symmetry in $\bl{v}$. For example, when $m=3$, the quantity $\f{\p_rv_z}{r}$ cannot be regarded as a component of the 2nd-order derivative tensor $\na^2\bl{v}$, and neither $\p_r\f{\p_rv_z}{r}$ nor $\p_z\f{\p_rv_z}{r}$ can be considered as a component of the 3rd-order derivative tensor $\na^3\bl{v}$. This creates problems when estimating them. Nevertheless, after a rigorous but tedious calculation, one can prove that:
\[
\begin{split}
\na^2\bl{v}\,:\,\bl{e_z}\otimes\bl{e_\theta}\otimes\bl{e_\theta}&=\f{\p_rv_z}{r}+\f{\p^2_\theta}{r^2}{v}_z\,;\\
\na^3\bl{v}\,:\,\bl{e_z}\otimes\bl{e_r}\otimes\bl{e_\theta}\otimes\bl{e_\theta}&=\p_r\big(\f{\p_rv_z}{r}\big)+\f{\p_\theta}{r^2}\big(\p_\theta\p_r{v}_z-\f{2\p_\theta{v}_z}{r}\big)\,;\\
\na^3\bl{v}\,:\,\bl{e_z}\otimes\bl{e_\theta}\otimes\bl{e_\theta}\otimes\bl{e_z}&=\p_z\big(\f{\p_rv_z}{r}\big)+\f{\p^2_\theta}{r^2}\p_z{v}_z\,,\\
\end{split}
\]
which suggests that:
\[
\begin{split}
\big\|\f{\p_rv_z}{r}+\f{\p^2_\theta}{r^2}{v}_z\big\|_{L^2}&\leq \|\bl{v}\|_{\dot{H}^2}\,;\\
\big\|\p_r\big(\f{\p_rv_z}{r}\big)+\f{\p_\theta}{r^2}\big(\p_\theta\p_r{v}_z-\f{2\p_\theta{v}_z}{r}\big)\big\|_{L^2}&+\big\|\p_z\big(\f{\p_rv_z}{r}\big)+\f{\p^2_\theta}{r^2}\p_z{v}_z\big\|_{L^2}\leq \|\bl{v}\|_{\dot{H}^3}\,.\\
\end{split}
\]
Therefore, by applying the aforementioned ``$\theta-$cancellation" property \eqref{ENS}, one settles this issue. In the current paper, a detailed structure of higher-order derivative tensor of a function in the cylindrical coordinate system will be shown, and this plays as a cornerstone while carrying out local $H^m$ energy estimate for any $3\leq m\in\mathbb{N}$.

The rest of this paper is organized as follows. To derive the energy the higher-order energy estimate of \eqref{(2.1)} in the cylindrical coordinates, some geometric preparation is carried out in Section \ref{PRE}: Detailed structure of the higher-order derivative tensor in the cylindrical coordinates is carefully studied. Also useful lemmas concerning interpolation inequalities and the commutator estimates are listed there. Finally, the main results will be proved in Section \ref{Main}.

At the end of this section, we show a list of notations that will appear throughout the paper:
\begin{itemize}
\item $C_{a,b,...}$ denotes a positive constant depending on $a,\,b,\,...$, which may be different from line to line. Likewise, we use $C_{0,...}$ to denote a constant that also depends on initial data.
\item $A\lesssim B$ means $A\leq CB$, and $A\simeq B$ denotes both $A\lesssim B$ and $B\lesssim A$.
\item $[\mathcal{A},\,\mathcal{B}]=\mathcal{A}\mathcal{B}-\mathcal{B}\mathcal{A}$ denotes the communicator of the operator $\mathcal{A}$ and the operator $\mathcal{B}$.
\item $\na$ stands for the classical gradient operator:
    \[
    \nabla^\mathfrak{H}:=\p_{x_1}^{h_1}\p_{x_2}^{h_2}\p_{x_3}^{h_3}.
    \]
    Here $\mathfrak{H}$ is a 3D multi-index such that $\mathfrak{H}=(h_1,h_2,h_3)$ with $h_1,h_2,h_3\in\mathbb{N}\cup\{0\}$ and $|\mathfrak{H}|=h_1+h_2+h_3$ .
\item $\bar{\na}^\mathfrak{L}$ denotes the axisymmetric gradient operator:
    \[
    \bar{\na}^\mathfrak{L}:=\p_z^{l_z}\p_r^{l_r}.
    \]
    Here $\mathfrak{L}$ is a 2D multi-index such that $\mathfrak{L}=(l_r,l_z)$ with $l_r,l_z\in\mathbb{N}\cup\{0\}$ and $|\mathfrak{L}|=l_r+l_z$ .
\item $\mD^\mathfrak{M}$ denotes the following compound gradient operator in cylindrical coordinates:
    \[
    \mD^\mathfrak{M}:=\p_z^{m_z}\p_r^{m_r}\Big(\f{\p_r}{r}\Big)^{m_c}.
    \]
    Here $\mathfrak{M}$ is a 3D multi-index such that $\mathfrak{M}=(m_c,m_r,m_z)$ with $m_c,m_r,m_z\in\mathbb{N}\cup\{0\}$ and $|\mathfrak{M}|=2m_c+m_r+m_z$ . Also we denote $\bar{\mathfrak{M}}=(m_r,m_z)$, with $|\bar{\mathfrak{M}}|=m_r+m_z$ .

\item We use standard notations for Lebesgue and Sobolev functional spaces in $\mathbb{R}^3$: For $1\leq p\leq\infty$ and $k\in\mathbb{N}$, $L^p$ denotes the Lebesgue space with norm
\[
\|f\|_{L^p}:=
\lt\{
\begin{aligned}
&\Big(\int_{\mathbb{R}^3}|f(x)|^pdx\Big)^{1/p},\quad 1\leq p<\infty,\\
&\mathop{ess sup}_{x\in\mathbb{R}^3}|f(x)|,\quad\quad\quad\quad p=\infty.\\
\end{aligned}
\rt.
\]

\item $H^m$ denotes the usual $L^2$-based Sobolev space with its norm
\[
\begin{split}
\|f\|_{H^m}:=&\sum_{0\leq|\mathfrak{H}|\leq m}\|\nabla^\mathfrak{H} f\|_{L^2}\,.\\
\end{split}
\]
\end{itemize}
\section{Preliminary}\label{PRE}
\subsection{Geometric properties of cylindrical coordinates}
The cylindrical coordinates system $(r,\theta,z)$ is defined by the mapping:
\[
\bl{X}:[0,\i)\times[0,2\pi)\times\mR\,\to\,\mathbb{R}^3\,,\q\q\bl{X}(r,\theta,z)=(r\cos\theta,r\sin\theta,z)\,.
\]
From this, we denote its covariant frame system as:
\[
\left\{
\begin{aligned}
&{\mf{E}}_r:=\p_r\bl{X}=(\cos\theta, \sin\theta, 0)=\bl{e_r};\\[1mm]
&{\mf{E}}_\theta:=\p_\theta\bl{X}=(-{r\sin\theta},r{\cos\theta}, 0)=r\bl{e_\theta};\\[1mm]
&{\mf{E}}_z:=\p_z\bl{X}=(0,0,1)=\bl{e_z}\,.
\end{aligned}
\right.
\]
And its related contravariant frame system is
\[
\left\{
\begin{split}
&{\mf{E}}^r=(\cos\theta, \sin\theta, 0)=\bl{e_r};\\[1mm]
&{\mf{E}}^\theta=(-\f{\sin\theta}{r},\f{\cos\theta}{r}, 0)=\f{\bl{e_\theta}}{r};\\[1mm]
&{\mf{E}}^z=(0,0,1)=\bl{e_z}\,.
\end{split}
\right.
\]
For our further calculations, we first recall the Christoffel symbol $\Gamma_{ij}^k$, which is defined by
\[
\Gamma_{ij}^k={\mf{E}}^i\cdot\p_j{\mf{E}}_k=-{\mf{E}}_i\cdot\p_j{\mf{E}}^k,\q i,j,k\in\{r,\theta,z\}\,,
\]
direct calculation shows
\be\label{CHRI}
\Gamma_{ij}^k=\left\{
\begin{array}{lll}
&-r\q\,\,, &\text{for }\,\,(i,j)=(\theta,\theta)\,\,, k=r\,\,;\\[2mm]
&\f{1}{r}\q\,\,\,\,\,, &\text{for }\,\,(i,j)=(r,\theta)\,\,\text{ or }\,\,(\theta,r),\,\, k=\theta\,\,;\\[2mm]
&0\q\,\,\,\,\,, &\text{otherwise}.
\end{array}
\right.
\ee

Using this, for any $f\,:\,\mathbb{R}^3\to\mathbb{R}$ being smooth enough, its gradient $\na f=(\p_{x_1}f,\p_{x_2}f,\p_{x_3}f)$ can be represented as
\[
\na f=\p_r f{\mf{E}}^r+\p_\theta f{\mf{E}}^\theta+\p_z f{\mf{E}}^z=\p_r f\bl{e_r}+\f{\p_\theta f}{r}\bl{e_\theta}+\p_z f\bl{e_z}\,.
\]
Generally, for any $n\in\mathbb{N}$ and $\iota_i\in\{r,\,\theta,\,z\}$, for $i=1,2,...,n,n+1$,
\be\label{NTHOR}
\begin{split}
(\na^{n+1}f)_{\iota_1,\iota_2,...,\iota_{n+1}}:=&\na^{n+1}f\,:\,{\mf{E}_{\iota_{1}}}\otimes{\mf{E}_{\iota_{2}}}\otimes\cdot\cdot\cdot\otimes{\mf{E}_{\iota_{n+1}}}\\
=&\p_{\iota_{n+1}}(\na^{n}f)_{\iota_1,\iota_2,...,\iota_{n}}-\sum_{s\in\{r,\theta,z\}}\sum_{i=1}^n\Gamma_{\iota_i\iota_{n+1}}^s(\na^nf)_{\iota_1,...,\iota_{i-1},\hat{\iota_i},s,\iota_{i+1},...,\iota_n}\,.
\end{split}
\ee

Meanwhile, if $f=f(r,z)$ be axially symmetric, then
\be\label{1st}
\begin{split}
\na f&=\p_rf{\mf{E}}^r+\p_zf{\mf{E}}^z=\p_rf \bl{e_r}+\p_zf\bl{e_z}\,.
\end{split}
\ee
Below we denote $(\na^2f)_{ij}:=\na^2f\,:\,{\mf{E}}_i\otimes{\mf{E}}_j$, for $i,j\in\{r,\theta,z\}$. Applying \eqref{CHRI} and \eqref{NTHOR}, one derives
\be\label{2nd}
\begin{split}
(\na^2f)_{ij}=&\,\p_j(\p_if)-\sum_{s\in\{r,\theta,z\}}\Gamma_{ij}^s\p_sf\\
=&\,\p_r^2 f\,{\mf{E}}^r\otimes{\mf{E}}^r+\p_z\p_r f\,{\mf{E}}^r\otimes{\mf{E}}^z+r\p_rf\,\,{\mf{E}}^\theta\otimes{\mf{E}}^\theta+\p_r\p_z f\,{\mf{E}}^z\otimes{\mf{E}}^r+\p^2_z f\,{\mf{E}}^z\otimes{\mf{E}}^z\\
=&\,\p_r^2 f\,\bl{{e}}_r\otimes\bl{{e}}_r+\p_z\p_r f\,\bl{{e}}_r\otimes\bl{{e}}_z+\f{\p_r}{r}f\,\,\bl{{e}}_\theta\otimes\bl{{e}}_\theta+\p_r\p_z f\,\bl{{e}}_z\otimes\bl{{e}}_r+\p^2_z f\,\bl{{e}}_z\otimes\bl{{e}}_z\,.
\end{split}
\ee
To calculate derivative tensor of $f$ with arbitrary order, we first show the component of derivative tensor with odd $\theta$ lower indexes must be zero.
\begin{lemma}\label{lemodd}
Given $\iota_i\in\{r,\,\theta,\,z\}$, for $i=1,2,...,n$. Define
\[
\chi_\theta[\bl{e_{\iota_1}},\,\bl{e_{\iota_2}},...,\bl{e_{\iota_n}}]:=\sum_{i=1}^n\bl{e_{\iota_i}}\cdot\bl{e_\theta}\,.
\]
Let $f=f(r,z)$ be an smooth axially symmetric function, then
\[
\na^nf\,:\,\bl{e_{\iota_1}}\otimes\bl{e_{\iota_2}}\otimes\cdot\cdot\cdot\otimes\bl{e_{\iota_n}}\equiv0
\]
if
\[
\chi_\theta[\bl{e_{\iota_1}},\,\bl{e_{\iota_2}},...,\bl{e_{\iota_n}}]\equiv 1\mod 2\,.
\]
\end{lemma}

\begin{proof}
From \eqref{1st} and \eqref{2nd} above, we know Lemma \ref{lemodd} holds for $n=1,2$. Now we prove the result for any $n\in\mathbb{N}$ by induction. Suppose the lemma holds for $n$, and we then consider the case with $n+1$. Indeed, given
\[
(\na^{n+1}f)_{\iota_1,\iota_2,...,\iota_{n+1}}:=\na^{n+1}f\,:\,{\mf{E}_{\iota_{1}}}\otimes{\mf{E}_{\iota_{2}}}\otimes\cdot\cdot\cdot\otimes{\mf{E}_{\iota_{n+1}}}\,,
\]
with
\[
\tilde{\chi}_\theta[{\mf{E}^{\iota_{1}}},\,{\mf{E}^{\iota_{2}}},...,{\mf{E}^{\iota_{n+1}}}]:=\sum_{i=1}^{n+1}{\mf{E}^{\iota_{i}}}\cdot{\mf{E}^\theta}
\]
being odd. Noticing that
\[
(\na^{n+1}f)_{\iota_1,\iota_2,...,\iota_{n+1}}=\p_{\iota_{n+1}}(\na^{n}f)_{\iota_1,\iota_2,...,\iota_{n}}-\sum_{s\in\{r,\theta,z\}}\sum_{i=1}^n\Gamma_{\iota_i\iota_{n+1}}^s(\na^nf)_{\iota_1,...,\iota_{i-1},\hat{\iota_i},s,\iota_{i+1},...,\iota_n}\,,
\]
we split our proof into the following three cases:\\[1mm]

\noindent\textbf{Case I: $\iota_{n+1}=z$.}

In this case, since $\Gamma_{\iota_iz}^s\equiv0$, we have
\[
(\na^{n+1}f)_{\iota_1,\iota_2,...,\iota_n,z}=\p_{z}(\na^{n}f)_{\iota_1,\iota_2,...,\iota_{n}}\,,
\]
and
\[
\tilde{\chi}_\theta[{\mf{E}^{\iota_{1}}},\,{\mf{E}^{\iota_{2}}},...\,,{\mf{E}^{\iota_{n}}}]=\tilde{\chi}_\theta[{\mf{E}^{\iota_{1}}},\,{\mf{E}^{\iota_{2}}},...\,,{\mf{E}^{\iota_{n+1}}}]\equiv 1\mod 2\,.
\]
Thus the Lemma is proved since we assume it holds for the $n$th order.\\[1mm]

\noindent\textbf{Case II: $\iota_{n+1}=r$.}

In this case, we write
\be\label{CAII}
\begin{split}
(\na^{n+1}f)_{\iota_1,\iota_2,...,\iota_{n},r}&=\p_{r}(\na^{n}f)_{\iota_1,\iota_2,...,\iota_{n}}-\sum_{s\in\{r,\theta,z\}}\sum_{i=1}^n\Gamma_{\iota_ir}^s(\na^nf)_{\iota_1,...,\iota_{i-1},\hat{\iota_i},s,\iota_{i+1},...,\iota_n}\\
&=\p_{r}(\na^{n}f)_{\iota_1,\iota_2,...,\iota_{n}}-\sum_{\iota_i=\theta}\Gamma_{\theta r}^\theta(\na^nf)_{\iota_1,...,\iota_{i-1},\hat{\theta},\theta,\iota_{i+1},...,\iota_n}\,.
\end{split}
\ee
Here the second equation follows because $\Gamma_{\iota_ir}^s\equiv0$, if $(s,\iota_i)\neq(\theta,\theta)$. Noticing
\[
\tilde{\chi}_\theta[{\mf{E}^{\iota_{1}}},...\,,{\mf{E}^{\iota_{i-1}}},\hat{{\mf{E}^{\theta}}},{\mf{E}^{\theta}},...\,,{\mf{E}^{\iota_{i+1}}},...\,,{\mf{E}^{\iota_{n}}}]=\tilde{\chi}_\theta[{\mf{E}^{\iota_{1}}},\,{\mf{E}^{\iota_{2}}},...\,,{\mf{E}^{\iota_{n}}}]\equiv1\mod 2\,,
\]
terms $(\na^nf)_{\iota_1,...,\iota_{i-1},\hat{\theta},\theta,\iota_{i+1},...,\iota_n}$ in the far right of \eqref{CAII} all vanish. This case is also proved by induction.\\[1mm]

\noindent\textbf{Case III: $\iota_{n+1}=\theta$.}

First, by induction, it is clear that $(\na^{n}f)_{\iota_1,\iota_2,...,\iota_{n}}$ is independent with $\theta$. This indicates
\be\label{thest}
\begin{split}
(\na^{n+1}f)_{\iota_1,\iota_2,...,\iota_{n},\theta}&=-\sum_{s\in\{r,\theta,z\}}\sum_{i=1}^n\Gamma_{\iota_i\theta}^s(\na^nf)_{\iota_1,...,\iota_{i-1},\hat{\iota_i},s,\iota_{i+1},...,\iota_n}\\
&=-\sum_{\iota_i=\theta}\Gamma_{\theta \theta}^r(\na^nf)_{\iota_1,...,\iota_{i-1},\hat{\theta},r,\iota_{i+1},...,\iota_n}-\sum_{\iota_i=r}\Gamma_{r\theta }^\theta(\na^nf)_{\iota_1,...,\iota_{i-1},\hat{r},\theta,\iota_{i+1},...,\iota_n}\,.
\end{split}
\ee
Here the second equation holds due to $\Gamma_{\iota_i\theta}^s\neq0$ only if $(s,\iota_i)=(r,\theta)$ or $(s,\iota_i)=(\theta,r)$. Observing that
\[
\tilde{\chi}_\theta[{\mf{E}^{\iota_{1}}},...\,,{\mf{E}^{\iota_{i-1}}},\hat{{\mf{E}^{\theta}}},{\mf{E}^{r}},...\,,{\mf{E}^{\iota_{i+1}}},...\,,{\mf{E}^{\iota_{n}}}]=\tilde{\chi}_\theta[{\mf{E}^{\iota_{1}}},\,{\mf{E}^{\iota_{2}}},...\,,{\mf{E}^{\iota_{n}}},\,{\mf{E}^{\theta}}]-2\,,
\]
and
\[
\tilde{\chi}_\theta[{\mf{E}^{\iota_{1}}},...\,,{\mf{E}^{\iota_{i-1}}},\hat{{\mf{E}^{r}}},{\mf{E}^{\theta}},...\,,{\mf{E}^{\iota_{i+1}}},...\,,{\mf{E}^{\iota_{n}}}]=\tilde{\chi}_\theta[{\mf{E}^{\iota_{1}}},\,{\mf{E}^{\iota_{2}}},...\,,{\mf{E}^{\iota_{n}}}]\,,
\]
one has both 
\[
\tilde{\chi}_\theta[{\mf{E}^{\iota_{1}}},...\,,{\mf{E}^{\iota_{i-1}}},\hat{{\mf{E}^{\theta}}},{\mf{E}^{r}},...\,,{\mf{E}^{\iota_{i+1}}},...\,,{\mf{E}^{\iota_{n}}}]
\] 
and 
\[
\tilde{\chi}_\theta[{\mf{E}^{\iota_{1}}},...\,,{\mf{E}^{\iota_{i-1}}},\hat{{\mf{E}^{r}}},{\mf{E}^{\theta}},...\,,{\mf{E}^{\iota_{i+1}}},...\,,{\mf{E}^{\iota_{n}}}]
\] 
are odd. By induction, all terms on far right of \eqref{thest} vanish. This finishes the proof of the Lemma.
\end{proof}

Since $\mathbb{R}^3$ is a flat manifold, the component $(\na^{n}f)_{\iota_1,\iota_2,...,\iota_{n}}$ is independent with the order of its lower indexes. With the help of Lemma \ref{lemodd}, given $f\in C^n(\mathbb{R}^3)$ be an axially symmetric scalar function, let $m_c,\,m_r,\,m_z\in\mathbb{N}$ with $2m_c+m_r+m_z=m$, the following character is well-defined:
\[
(\na^mf)_{2m_c\theta,m_rr,m_zz}:=(\na^mf)\,:\,\un{{\mf{E}_\theta}\otimes\cdot\cdot\cdot\otimes{\mf{E}_\theta}}_{2m_c\,\,\text{times}}\otimes\un{{\mf{E}_r}\otimes\cdot\cdot\cdot\otimes{\mf{E}_r}}_{m_r\,\,\text{times}}\otimes\un{{\mf{E}_z}\otimes\cdot\cdot\cdot\otimes{\mf{E}_z}}_{m_z\,\,\text{times}}\,.
\]
Now we are ready for the characterization of the derivative tensors of a scalar function in the cylindrical coordinates:
\begin{proposition}\label{THM3}
Given $f\in C^m(\mathbb{R}^3)$ be an axially symmetric scalar function. Let $\mathfrak{M}=(m_c,\,m_r,\,m_z)\in(\mathbb{N}\cup\{0\})^3$ be a multi-index and $|\mathfrak{M}|=2m_c+m_r+m_z=m$. Then
\be\label{MM}
\begin{split}
(\na^mf)_{2m_c\theta,m_rr,m_zz}&=(2m_c-1)!!r^{2m_c}\,\p_z^{m_z}\p_r^{m_r}\Big(\f{\p_r}{r}\Big)^{m_c}f\,.
\end{split}
\ee
This indicates that
\[
(\na^mf)\,:\,\un{\bl{e_\theta}\otimes\cdot\cdot\cdot\otimes\bl{e_\theta}}_{2m_c\,\,\text{times}}\otimes\un{\bl{e_r}\otimes\cdot\cdot\cdot\otimes\bl{e_r}}_{m_r\,\,\text{times}}\otimes\un{\bl{e_z}\otimes\cdot\cdot\cdot\otimes\bl{e_z}}_{m_z\,\,\text{times}}=(2m_c-1)!!\mD^{\mathfrak{M}}f\,,
\]
and
\[
|\na^mf|^2\simeq\sum_{|\mathfrak{M}|=m}\Big|\mD^{\mathfrak{M}}f\Big|^2\,.
\]
\end{proposition}
\begin{proof}
We first show
\be\label{MM1}
(\na^{2m_c}f)_{2m_c\theta}=(2m_c-1)!!r^{2m_c}\Big(\f{\p_r}{r}\Big)^{m_c}f\,,
\ee
by induction. The case $m_c=1$
\[
(\na^{2}f)_{\theta\theta}=r^{2}\Big(\f{\p_r}{r}\Big)f
\]
is proved in the equation \eqref{2nd}, and we assume the case of $m_c=m'-1$ also holds. For the case of $m_c=m'$, the identity \eqref{NTHOR} shows
\be\label{Lt1}
(\na^{2m'}f)_{2m'\theta}=-(2m'-1)\Gamma_{\theta\theta}^r(\na^{2m'-1}f)_{(2m'-2)\theta,r}=(2m'-1)r(\na^{2m'-1}f)_{(2m'-2)\theta,r}\,.
\ee
Calculating one more order covariant derivative with $r$, one has
\be\label{Lt2}
(\na^{2m'-1}f)_{(2m'-2)\theta,r}=\p_r(\na^{2m'-2}f)_{(2m'-2)\theta}-(2m'-2)\Gamma_{\theta r}^\theta(\na^{2m'-2}f)_{(2m'-2)\theta}\,.
\ee
Substituting \eqref{Lt2} in \eqref{Lt1}, and noticing we have assumed \eqref{MM1} holds for $m_c=m'-1$, one derives
\[
\begin{split}
(\na^{2m'}f)_{2m'\theta}&=(2m'-1)r\big(\p_r-\f{2m'-2}{r}\big)(\na^{2m'-2}f)_{(2m'-2)\theta}\\
&=(2m'-1)r\big(\p_r-\f{2m'-2}{r}\big)\bigg[(2m'-3)!!r^{2m'-2}\Big(\f{\p_r}{r}\Big)^{m'-1}f\bigg]\\
&=(2m'-1)!!r^{2m'}\Big(\f{\p_r}{r}\Big)^{m'}f\,,
\end{split}
\]
which shows \eqref{MM1} also holds for $m'$-case. This concludes the validity of \eqref{MM1}.

Then, for $1\leq k\leq m_r$, by \eqref{NTHOR}, one notices that
\be\label{MM2}
\begin{split}
(\na^{2m_c+k}f)_{2m_c\theta,\,kr}&=\p_{r}(\na^{2m_c+k-1}f)_{2m_c\theta,\,(k-1)r}-2m_c\Gamma_{\theta r}^\theta(\na^{2m_c+k-1}f)_{2m_c\theta,\,(k-1)r}\\
&=\big(\p_r-\f{2m_c}{r}\big)(\na^{2m_c+k-1}f)_{2m_c\theta,\,(k-1)r}\\
&=...=\big(\p_r-\f{2m_c}{r}\big)^k(\na^{2m_c}f)_{2m_c\theta}\,.
\end{split}
\ee
Since the identity
\[
\big(\p_r-\f{2m_c}{r}\big)(r^{2m_c}g)=r^{2m_c}\p_rg
\]
holds for any smooth function $g$, one derives
\[
(\na^{2m_c+m_r}f)_{2m_c\theta,m_rr}=(2m_c-1)!!r^{2m_c}\,\p_r^{m_r}\Big(\f{\p_r}{r}\Big)^{m_c}f
\]
by applying \eqref{MM1} and \eqref{MM2} and iterating with $k=1,\,2,\,3,\,...,\,m_r$. Finally, \eqref{MM} follows by taking $m_z$-th order $z$ derivatives.
\end{proof}

Here follows a direct corollary of Proposition \ref{THM3}, which represents a component of the derivative tensor of a non-axisymmetric function. Generally speaking, the only difference with the axially symmetric case is an extra term with $\theta$-derivative.

\begin{corollary}\label{COR255}
Given $g\in C^{m}(\mathbb{R}^3)$ be a scalar function. Let $\iota_i\in\{r,z,\theta\}$, for $i=1,2,...,m$, and denote $m_r={\chi}_r[\bl{{e}_{\iota_{1}}},\,\bl{{e}_{\iota_{2}}},...,\bl{{e}_{\iota_{m}}}]$, $m_z={\chi}_z[\bl{{e}_{\iota_{1}}},\,\bl{{e}_{\iota_{2}}},...,\bl{{e}_{\iota_{m}}}]$, $m_c={\chi}_\theta[\bl{{e}_{\iota_{1}}},\,\bl{{e}_{\iota_{2}}},...,\bl{{e}_{\iota_{m}}}]/2$. Then
\[
\begin{split}
&(\na^{m}g)_{\iota_1,\iota_2,...,\iota_{m}}:=\na^{m}g\,:\,{\mf{E}_{\iota_{1}}}\otimes{\mf{E}_{\iota_{2}}}\otimes\cdot\cdot\cdot\otimes{\mf{E}_{\iota_{m}}}\\[1mm]
=&\left\{
\begin{aligned}
&(2m_c-1)!!r^{2m_c}\,\p_z^{m_z}\p_r^{m_r}\Big(\f{\p_r}{r}\Big)^{m_c}g+\p_\theta F_{\iota_1,\iota_2,...,\iota_{m}}(r,z,\theta)\,,&\q\text{if } m_c\text{ is an integer;}\\
&\p_\theta F_{\iota_1,\iota_2,...,\iota_{m}}(r,z,\theta)\,,&\q\text{otherwise.}\\
\end{aligned}
\right.
\end{split}
\]
Here $F_{\iota_1,\iota_2,...,\iota_{m}}(r,z,\theta)$ is a scalar function depending only on its lower indexes. Or equivalently, we write the above equation in the unit coordinate system $\{\bl{e_r},\,\bl{e_\theta},\,\bl{e_z}\}$ :
\[
\begin{split}
&\na^{n}g\,:\,\bl{{e}_{\iota_{1}}}\otimes\bl{{e}_{\iota_{2}}}\otimes\cdot\cdot\cdot\otimes\bl{{e}_{\iota_{m}}}\\
=&\left\{
\begin{aligned}
&(2m_c-1)!!\,\p_z^{m_z}\p_r^{m_r}\Big(\f{\p_r}{r}\Big)^{m_c}g+r^{-2m_c}\p_\theta F_{\iota_1,\iota_2,...,\iota_{m}}(r,z,\theta)\,,&\q\text{if } m_c\text{ is an integer;}\\
&{r^{-2m_c}}{\p_\theta}F_{\iota_1,\iota_2,...,\iota_{m}}(r,z,\theta)\,,&\q\text{otherwise.}\\
\end{aligned}
\right.
\end{split}
\]
\end{corollary}

\begin{proof}
Proposition \ref{THM3} considered the case without any $\theta$ derivative. Since all the non-zero Christoffel symbols
\[
\Gamma_{\theta\theta}^r=-r,\q\Gamma_{r\theta}^\theta=\Gamma_{\theta r}^\theta=\f{1}{r}
\]
are all independent with $\theta$, once a $\theta$-derivative is created, it can always be ``shifted" to the front. We gather all these terms in $\p_\theta F_{\iota_1,\iota_2,...,\iota_{m}}(r,z,\theta)$. This finishes the proof of the corollary.
\end{proof}

In the cylindrical coordinates, the $x_3$ direction is indeed flat. Thus for a smooth enough 3D vector field $\bl{v}=v_r\bl{e_r}+{v_\theta}\bl{e_\theta}+v_z\bl{e_z}$, given $2m_c+m_r+m_z=m\in\mathbb{N}$, one has
\[
(2m_c-1)!!\,\p_z^{m_z}\p_r^{m_r}\Big(\f{\p_r}{r}\Big)^{m_c}v_z+r^{-2m_c}\p_\theta F_{\iota_1,\iota_2,...,\iota_{m}}(r,z,\theta)
\]
is a component of tensor $\na^m\bl{v}$ for some scalar function $F_{\iota_1,\iota_2,...,\iota_{m}}(r,z,\theta)$. Meanwhile, the following lemma shows $\bar{\na}^\mathfrak{L}v_r$ and $\bar{\na}^\mathfrak{L}v_z$ $(|\mathfrak{L}|=n)$ are also a component of the derivative tensor $\na^n\bl{v}$ :

\begin{lemma}\label{LEMa}
Let $\mathfrak{L}$ be a 2D multi-index such that $|\mathfrak{L}|=n$. Then $\bar{\na}^\mathfrak{L}v_r$ and $\bar{\na}^\mathfrak{L}v_z$ are components of derivative tensor $\na^n\bl{v}$.
\end{lemma}
\begin{proof} 
For $\iota_1$, $\iota_2$,...,$\iota_{n+1}=r$ or $z$, a direct calculation of tensor derivatives follows
\[
\begin{split}
(\na^{n}\bl{v})\,:\,\bl{{e}_{\iota_1}}\otimes\bl{{e}_{\iota_2}}\otimes\cdot\cdot\cdot\,\otimes\bl{{e}_{\iota_{n+1}}}&=(\na^{n}\bl{v})\,:\,{\mf{E}_{\iota_1}}\otimes{\mf{E}_{\iota_2}}\otimes\cdot\cdot\cdot\,\otimes{\mf{E}_{\iota_{n+1}}}\\
&=\p_{\iota_{n+1}}(\na^{n-1}\bl{v})_{\iota_1,...,\iota_{n}}-\sum_{s\in\{r,\theta,z\}}\sum_{j=1}^{n}\Gamma_{\iota_j\,\iota_n}^s(\na^{n-1}\bl{v})_{\iota_1,...,\hat{\iota_j},s,...,\iota_{n}}\,.
\end{split}
\]
Since $\iota_i\neq\theta$, for any $i=1,2,...,n+1$, one has $\Gamma_{\iota_j\,\iota_{n+1}}^s=0$. This indicates
\[
(\na^n\bl{v})_{\iota_1,...,\iota_n}=\p_{\iota_{n+1}}(\na^{n-1}\bl{v})_{\iota_1,...,\iota_{n}}\,.
\]
Thus a direct induction follows
\[
(\na^{n}\bl{v})\,:\,\bl{{e}_{\iota_1}}\otimes\bl{{e}_{\iota_2}}\otimes\cdot\cdot\cdot\,\otimes\bl{{e}_{\iota_{n+1}}}=\p_{\iota_{n+1}}\cdot\cdot\cdot\p_{\iota_2}(\bl{v}\cdot\bl{{e}_{\iota_1}})\,,
\]
which concludes the Lemma.
\end{proof}

\subsection{Other useful lemmas}
Two well-known lemmas will be listed in this section with out detailed proof. First, let us introduce the well-known $Gagliardo-Nirenberg$ interpolation inequality. We refer readers to \cite{GN} for a detailed proof.
\begin{lemma}[Gagliardo-Nirenberg]\label{LEMGN}
Given $q,r\in[1,\i]$ and $j,m\in\mathbb{N}\cup\{0\}$ with $j\leq m$. Suppose that $f\in L^q(\mathbb{R}^3)$, $\na^mf\in L^r(\mathbb{R}^3)$ and there exists a real number $\al\in[j/m,1]$ such that
\[
\frac{1}{p}=\frac{j}{3}+\al\Big(\frac{1}{r}-\frac{m}{3}\Big)+\frac{1-\al}{q}.
\]
Then $f\in\dot{W}^{j,3}(\mathbb{R}^d)$ and there exists a constant $C>0$ such that
\[
\|\na^jf\|_{L^p(\mathbb{R}^3)}\leq C\|\na^m f\|^\al_{L^r(\mathbb{R}^3)}\|f\|^{1-\al}_{L^q(\mathbb{R}^3)},
\]
except the following two cases:
\begin{itemize}
\item[I.] $j=0$, $mr<d$ and $q=\infty$; (In this case it is necessary to assume also that either $|u|\to 0$ at infinity, or $u\in L^s(\mathbb{R}^d)$ for some $s<\infty$.)

\item[II.] $1<r<\infty$ and $m-j-d/r\in\mathbb{N}$. (In this case it is necessary to assume also that $\alpha<1$.)
\end{itemize}
\end{lemma}

\qed

Now we focus on the following estimates of a triple product form which will be frequently applied in the final proof.
\begin{lemma}\label{LEMET1}
Let $m\in\mathbb{N}$ and $m\geq 2$, $f,g,k\in C^\infty_0(\mathbb{R}^3)$. The following estimate holds:
\be\label{E1}
\begin{split}
\bigg|\int_{\mathbb{R}^3}[\nabla^m,\,f\cdot\nabla]g\nabla^m kdx\bigg|\leq&\,C\,\|\nabla^{m}(f,g,k)\|_{L^2}^2\|\nabla (f,\,g)\|_{L^\infty}.\\
\end{split}
\ee
\end{lemma}

\begin{proof} We apply H\"{o}lder inequality, one derives
\be\label{2.77777}
\bigg|\int_{\mathbb{R}^3}[\nabla^m,\,f\cdot\nabla]g\nabla^m kdx\bigg|\leq\|[\nabla^m,\,f\cdot\nabla]g\|_{L^{2}}\|\nabla^mk\|_{L^2}.
\ee
Due to the commutator estimate by Kato-Ponce \cite{KP:1988CPAM}, it follows that
\be\label{2.88888}
\|[\nabla^m,\,f\cdot\nabla]g\|_{L^{2}}\leq C\big(\|\nabla f\|_{L^\infty}\|\nabla^mg\|_{L^2}+\|\nabla g\|_{L^\infty}\|\nabla^m f\|_{L^2}\big).
\ee
Then \eqref{E1} follows from substituting \eqref{2.88888} in \eqref{2.77777}.
\end{proof}


\section{Proof of main results}\label{Main}
\subsection{Fundamental energy estimates and ${L^p}$ conservation of ${\mH}$}\label{SEC3.1}

\q\ At the beginning, the following Lemma states fundamental estimates of the system \eqref{Hall}:
\begin{lemma}[Fundamental energy estimate]\label{FUND}
Let $(\bl{v},\bl{h})\in H^3$ be the solution of \eqref{Hall}, then we have:
\be\label{EFF}
\|(\bl{v},\bl{h})(t,\cd)\|_{L^2}^2\leq \|(\bl{v}_0,\bl{h}_0)\|_{L^2}^2\,.
\ee
\end{lemma}
\begin{proof}
Inequality \eqref{EFF} follows from standard $L^2$ energy estimate of the system \eqref{Hall}, we omit the details here.
\end{proof}

\begin{lemma}[$L^p$ conservation of $\mH$]\label{LEMFUND}
Define $\mH:=\f{h_\theta}{r}$. Let $(\bl{v},\bl{h})\in H^3$ be the solution of \eqref{Hall}, then we have: For $p\in[2,\infty]$ and $t\in(0,\i)$,
\be\label{4.2}
\|\mH(t,\cdot)\|_{L^p}=\|\mH_0\|_{L^p}.
\ee
\end{lemma}

\begin{proof}
Recall $\eqref{EHH}$:
\[
\p_t\mH+(v_r\p_r+v_z\p_z)\mH-\mH\f{\p_\theta v_\theta}{r}-2\mH\p_z\mH=0.
\]
For any $p\geq2$, multiplying it by $p\mH|\mH|^{p-2}$ and integrating over $\mathbb{R}^3$, noting that
\[
p\int_{\mR^3}(v_r\p_r+v_z\p_z)\mH\cdot\mH|\mH|^{p-2}dx=\int_{\mR^3}\bl{v}\cdot\nabla|\mH|^pdx=-\int_{\mR^3}(\text{div }\bl{v})|\mH|^pdx=0,
\]
and
\[
p\int_{\mR^3}\f{\p_\theta v_\theta}{r}|\mH|^{p}dx=-p\int_{\mR^3}\f{ v_\theta}{r}\p_\theta|\mH|^{p}dx=0\,,
\]
one arrives
\[\begin{split}
\f{d}{dt}\|\mH(t,\cdot)\|_{L^p}^p&=\f{2}{p+1}\int_{\mathbb{R}^3}\p_z\big(\mH|\mH|^{p}\big)dx=0.\\
\end{split}\]
Integrating over $(0,t)$, one derives \eqref{4.2} for $p<\i$. The $L^\i$ case follows by performing $p\to\i$.
\end{proof}

\subsection{Higher-order estimates of $(\bl{v},\bl{h})$}

The following energy estimate in this subsection is relatively classical. Given $3\leq m\in\mathbb{N}$, applying $\na^m$ \eqref{Hall}$_{12}$, and performing the $L^2$ inner product of the resulting equations with $\na^m \bl{v}$ and $\na^m \bl{h}$ respectively, noting that
\[
\int_{\mR^3}\bl{h}\cdot\nabla\nabla^m\bl{h}\cdot\nabla^m\bl{v}dx+\int_{\mR^3}\bl{h}\cdot\nabla\na^m\bl{v}\cdot\na^m\bl{h}dx=0,
\]
and
\[
\na\times(\na\times\bl{h}\times\bl{h})=-2\mH\p_zh_\theta\bl{e_\theta}=-2\mH\p_z\bl{h}\,,
\]
we can obtain
\be\label{4.111}
\bali
&\f{1}{2}\f{d}{dt}\|\na^m(\bl{v},\bl{h})(t,\cdot)\|^2_{L^2}\\
=&-\un{\int_{\bR^3}[\nabla^m,\,\bl{v}\cdot\nabla]\bl{v}\cdot\na^m \bl{v} dx}_{I_1}+\un{\int_{\bR^3}[\na^m,\,\bl{h}\cdot\na]\bl{h}\cdot\na^m\bl{v}dx}_{I_2}-\un{\int_{\bR^3}[\na^m,\,\bl{v}\cdot\na]\bl{h}\cdot\na^m \bl{h}dx}_{I_3}\\
&+\un{\int_{\bR^3}[\nabla^m,\,\bl{h}\cdot\nabla]\bl{v}\cdot\na^m\bl{h}dx}_{I_4}+\un{2\int_{\bR^3}\nabla^m(\mH\p_z\bl{h})\cdot\nabla^m \bl{h}dx}_{I_5}\,.
\eali
\ee
Applying \eqref{E1} in Lemma \ref{LEMET1}, we have $I_1$--$I_4$ satisfy
\be\label{4.22}
I_j\,\lesssim\|\nabla (\bl{v}, \bl{h})(t,\cdot)\|_{L^\infty}\|\nabla^{m}(\bl{v},\bl{h})(t,\cdot)\|_{L^2}^2,\quad\forall j=1,2,3,4.
\ee
For term $I_5$, which arises from the Hall effect, we split it into two parts
\be\label{4.330}
I_5=\un{2\int_{\mR^3}\mH\p_z\na^m\bl{h}\cdot\na^m \bl{h}dx}_{I_{51}}+\un{2\int_{\mR^3}[\na^m,\mH\p_z]\bl{h}\cdot\na^m \bl{h}dx}_{I_{52}}.
\ee
Using integration by parts and H\"older's inequality, one arrives at
\be\label{4.331}
|I_{51}|=\Big|\int_{\mR^3}\p_z\mH|\na^m \bl{h}|^2dx\Big|\leq\|\p_z\mH(t,\cd)\|_{L^\i}\|\na^m \bl{h}(t,\cd)\|_{L^2}^2.
\ee
Using Lemma \ref{LEMET1}, one can handle $I_{52}$ in the following way
\be\label{4.332}
|I_{52}|\les \|\na(\bl{h},\mH)(t,\cd)\|_{L^\i}\|\na^m(\bl{h},\mH)(t,\cd)\|_{L^2}^2.
\ee
Substituting \eqref{4.332} and \eqref{4.331} to \eqref{4.330}, and then insert it together with \eqref{4.22} into the right hand side of \eqref{4.111}, one concludes
\[
\f{d}{dt}\|\na^m(\bl{v},\bl{h})(t,\cdot)\|^2_{L^2}\les\|\na(\bl{v},\bl{h},\mH)(t,\cd)\|_{L^\i}\|\na^m(\bl{v},\bl{h},\mH)(t,\cd)\|_{L^2}^2.
\]
Integrating on the temporal variable over $(0,t)$ and recalling the fundamental energy estimate of $(\bl{v},\bl{h})$ in Lemma \ref{FUND}, one concludes that
\be\label{ENVH}
\|(\bl{v},\bl{h})(t,\cdot)\|^2_{H^m}\leq\|(\bl{v}_0,\bl{h}_0)(t,\cdot)\|^2_{H^m}+C\int_0^t\|(\bl{v},\bl{h},\mH)(t,\cd)\|_{H^m}^3dt\,.
\ee
To close the above estimate, we will proceed with the required estimates of $\mH$ in the following subsection.
\subsection{Higher-order estimates of $\mH$}
Due to the lack of axial symmetry of $(\bl{v},\bl{h})$, here one extra term $\mH\f{\p_\theta v_\theta}{r}$ arises in the equation of $\mH$
\[
\p_t\mH+(v_r\p_r+v_z\p_z)\mH-\un{\mH\f{\p_\theta v_\theta}{r}}_{\text{Main difficulty}}-2\mH\p_z\mH=0\,.
\]
This causes a major difficulty in deriving the required higher order energy estimate of $\mH$. Indeed, if we roughly treat this term as $\mH\na\bl{v}$ and performing the $L^2$ inner product of $\na^m\mH$, one needs an estimate of the $(m+1)-$th order norm of $\bl{v}$ or $\mH$ to bound the following integral
\[
\int_{\bR^3}\na^m\mH\na^m(\mH\na\bl{v})dx,
\]
while we are only allowed to apply terms with $m-$th order. This results in an endless chain and it seems impossible to settle down the local well-posedness of \eqref{Hall} in a Sobolev space.

However, the $\theta$-derivative and the axial symmetry of $\mH$ inspire us to cancel this trouble term instead of to control it. More precisely, given $f=f(r,z)$ an axisymmetric function and $m\in\mathbb{N}$, it is clear that
\be\label{Van1}
\int_{\bR^3}f(r,z)\bar{\na}^m\Big(\mH\f{\p_\theta v_\theta}{r}\Big)dx=\int_{\bR}\int_0^\i\int_0^{2\pi}\f{d}{d\theta}\bigg(f(r,z)\bar{\na}^m\Big(\mH\f{v_\theta}{r}\Big)\bigg)d\theta rdrdz=0\,.
\ee
This observation seems to settle down the aforementioned trouble, but one problem is still alive: In fact, despite $\mH$ is axially symmetric, its gradient $\nabla^m\mH$ may not. A direct example is
\[
\p_{x_1}\mH=\cos\theta \p_r\mH\,.
\]
Thus to apply a vanishing property as \eqref{Van1}, results for structure of higher-order derivative tensor  $\na^m\mH$ and $\na^m\bl{v}$ in the cylindrical coordinates in Section \ref{PRE} will be applied. In the following, we will perform $L^2$-based energy estimates of $\na^m\mH$, and the main trouble term $\f{\p_\theta v_\theta}{r}\mH$ will always keep silent with the help of Proposition \ref{THM3}, Corollary \ref{COR255} and Lemma \ref{LEMa}.

Before that, one needs the following proposition which represents the commutator of $\mathcal{D}^\mathfrak{M}$ and $\bl{u}\cdot\bar{\nabla}$ :
\begin{proposition}\label{PROP33}
Given $f=f(r,z)$ and $g=g(r,z)$ are smooth axisymmetric scalar functions, let $\bl{u}=u_r(r,z,\theta)\bl{e_r}+u_z(r,z,\theta)\bl{e_z}+u_\theta(r,z,\theta)\bl{e_\theta}$ be a smooth divergence-free vector field. Then it holds that
\be\label{MMM0}
\begin{split}
\big[\mathcal{D}^\mathfrak{M}\,,\,g\p_z\big]f=&\sum_{1\leq|\mathfrak{N}|\leq |\mathfrak{M}|}C_{1,\,\mathfrak{N}}\mathcal{D}^{\mathfrak{N}}g\mathcal{D}^{\mathfrak{M}-\mathfrak{N}}\p_zf\,;
\end{split}
\ee
\be\label{MMM}
\begin{split}
\big[\mathcal{D}^\mathfrak{M}\,,\,\bl{u}\cdot\bar{\na}\big]f=&\sum_{1\leq|\mathfrak{N}|\leq |\mathfrak{M}|}C_{2,\,\mathfrak{N}}\mathcal{D}^{\mathfrak{N}}u_z\mathcal{D}^{\mathfrak{M}-\mathfrak{N}}\p_zf\\
&+\sum_{0\leq|\mathfrak{N}|\leq |\mathfrak{M}|-1}C_{3,\,\mathfrak{N}}\mathcal{D}^\mathfrak{N}\Big(\f{\p_\theta u_\theta}{r}+\p_zu_z\Big)\mathcal{D}^{\mathfrak{M}-\mathfrak{N}}f\\
&+\sum_{1\leq|\mathfrak{L}|\leq |\bar{\mathfrak{M}}|}C_{4,\,\mathfrak{L}}\bar{\na}^{\mathfrak{L}}u_r\p_r\mathcal{D}^{\mathfrak{M}-(0,l_r,l_z)}f\,.
\end{split}
\ee
Here $\bl{u}\cdot\bar{\na}=u_r\p_r+u_z\p_z$, while the coefficients $C_{i,\,\mathfrak{N}}$ for $i=1,2,3$ and $C_{4,\,\mathfrak{L}}$ are integers.
\end{proposition}

\begin{proof}
We only prove \eqref{MMM} since the first one is relatively more transparent. The proof is carried out with a compound induction on multi-index $\mathfrak{M}=(m_c,m_r,m_z)$. First we consider the case with $m_c=1$, $m_r=m_z=0$. Direct calculation shows
\[
\begin{split}
\big[\f{\p_r}{r}\,,\,\bl{u}\cdot\bar{\na}\big]f=&\frac{1}{r} \partial_r\bl{u}\cdot \bar{\nabla}f+\un{\frac{1}{r}\bl{u} \cdot \bar{\nabla}\p_rf}_{M_1}-\bl{u}\cdot\bar{\nabla}\frac{\p_rf}{r}\,.
\end{split}
\]
Noticing that
\[
M_1=\bl{u}\cd\bar{\na}\f{\p_rf}{r}+\f{u_r\p_rf}{r^2},
\]
one concludes
\be\label{al31}
\begin{split}
\big[\f{\p_r}{r}\,,\,\bl{u}\cdot\bar{\na}\big]f=&\Big(\partial_r u_r+\frac{u_r}{r}\Big) \frac{\partial_r f}{r}+\frac{\partial_r u_z}{r}\partial_z f=\frac{\partial_r u_z}{r}\partial_zf-\Big(\f{\p_\theta}{r}u_\theta+\p_zu_z\Big)\f{\p_rf}{r}\,.
\end{split}
\ee
Here we have applied the divergence property of $\bl{u}$ in the last equality. This shows the validity of \eqref{MMM} for $\mathfrak{M}=(1,0,0)$.

Given $m_c\in\mathbb{N}$, we \textbf{claim} the following identity holds:
\be\label{ASSm}
\begin{split}
\big[\big(\f{\p_r}{r}\big)^{m_c}\,,\,\bl{u}\cdot\bar{\na}\big]f=&\sum_{j=1}^{m_c}C_{1,j}\Big(\f{\p_r}{r}\Big)^ju_z\Big(\f{\p_r}{r}\Big)^{m_c-j}\p_zf\\
&+\sum_{j=0}^{m_c-1}C_{2,j}\Big(\f{\p_r}{r}\Big)^{j}\Big(\f{\p_\theta u_\theta}{r}+\p_zu_z\Big)\Big(\f{\p_r}{r}\Big)^{m_c-j}f\,.
\end{split}
\ee
Here and below, $C_{1,j}$ and $C_{2,j}$ are integers, which may change from line to line. The proof is carried out by induction. While case $m_c=1$ is already shown in \eqref{al31}, now we assume the validity of \eqref{ASSm}. By the direct calculation
\[
\begin{split}
\Big[\Big(\f{\p_r}{r}\Big)^{m_c+1}\,,\,\bl{u}\cdot\bar{\na}\Big]f=&\,\frac{\partial _r}{r}\bigg(\Big[\Big(\frac{\partial}{r}\Big)^{m_c}, \bl{u} \cdot \bar{\nabla}\Big]f+\bl{u} \cdot \bar{\nabla}\Big(\frac{\partial_r}{r}\Big)^{m_c}f\bigg)-\bl{u} \cdot \bar{\nabla}\Big(\frac{\partial_r}{r}\Big)^{m_c+1}f\\
=&\,\frac{\partial_r}{r}\Big[\Big(\frac{\partial}{r}\Big)^{m_c}, \bl{u} \cdot \bar{\nabla}\Big]f+\frac{\partial_r\bl{u}}{r} \cdot \bar{\nabla}\Big(\frac{\partial_r}{r}\Big)^{m_c}f+\un{\frac{1}{r} \bl{u}\cdot \bar{\nabla} \partial_r\Big(\frac{\partial_r}{r}\Big)^{m_c}f}_{M_2}\\
&\,-\bl{u} \cdot \bar{\nabla}\Big(\frac{\partial_r}{r}\Big)^{m_c+1}f\,.\\
\end{split}
\]
Since
\[
M_2=\bl{u}\cdot \bar{\nabla}\Big(\frac{\partial_r}{r}\Big)^{m_c+1} f+\frac{u_r}{r}\Big(\frac{\partial_r}{r}\Big)^{m_c+1} f,
\]
one deduces that by the divergence-free property of $\bl{u}$:
\[
\begin{split}
\Big[\Big(\f{\p_r}{r}\Big)^{m_c+1}\,,\,\bl{u}\cdot\bar{\na}\Big]f=&\frac{\partial_r}{r}\Big[\Big(\frac{\partial}{r}\Big)^{m_c}, \bl{u} \cdot \bar{\nabla}\Big]f+\Big(\p_ru_r+\f{u_r}{r}\Big)\Big(\frac{\partial_r}{r}\Big)^{m_c+1} f+\frac{\partial_ru_z}{r}\p_z\Big(\frac{\partial_r}{r}\Big)^{m_c}f\\
=&\frac{\partial_r}{r}\Big[\Big(\frac{\partial}{r}\Big)^{m_c}, \bl{u} \cdot \bar{\nabla}\Big]f+\frac{\partial_ru_z}{r}\p_z\Big(\frac{\partial_r}{r}\Big)^{m_c}f\\
&-\Big(\f{\p_\theta}{r}u_\theta+\p_zu_z\Big)\Big(\frac{\partial_r}{r}\Big)^{m_c+1} f\,.
\end{split}
\]
Substituting \eqref{ASSm} in the above equation, one arrives at
\[
\begin{split}
&\Big[\Big(\f{\p_r}{r}\Big)^{m_c+1}\,,\,\bl{u}\cdot\bar{\na}\Big]f\\
=&\sum_{j=1}^{m_c+1}C_{1,j}\Big(\f{\p_r}{r}\Big)^ju_z\Big(\f{\p_r}{r}\Big)^{m_c-j+1}\p_zf+\sum_{j=0}^{m_c}C_{2,j}\Big(\f{\p_r}{r}\Big)^{j}\Big(\f{\p_\theta u_\theta}{r}+\p_zu_z\Big)\Big(\f{\p_r}{r}\Big)^{m_c-j+1}f\\
&+\frac{\partial_ru_z}{r}\p_z\Big(\frac{\partial_r}{r}\Big)^{m_c}f-\Big(\f{\p_\theta}{r}u_\theta+\p_zu_z\Big)\Big(\frac{\partial_r}{r}\Big)^{m_c+1}f\\
=&\sum_{j=1}^{m_c+1}C_{1,j}\Big(\f{\p_r}{r}\Big)^ju_z\Big(\f{\p_r}{r}\Big)^{m_c-j+1}\p_zf+\sum_{j=0}^{m_c}C_{2,j}\Big(\f{\p_r}{r}\Big)^{j}\Big(\f{\p_\theta u_\theta}{r}+\p_zu_z\Big)\Big(\f{\p_r}{r}\Big)^{m_c-j+1}f\,.
\end{split}
\]
This implies the validity of the claim, and thus proves \eqref{MMM} for $\mathfrak{M}=(m_c,0,0)$, $\forall m_c\in\mathbb{N}$.

Now we carry out the induction of $m_r$.  We assume that
\be\label{ASSk}
\begin{split}
&\Big[\p_r^{m_r}\Big(\f{\p_r}{r}\Big)^{m_c}\,,\,\bl{u}\cdot\bar{\na}\Big]f\\
=\,&\p_r^{m_r}\Bigg\{\sum_{j=1}^{m_c+1}C_{1,j}\Big(\f{\p_r}{r}\Big)^ju_z\Big(\f{\p_r}{r}\Big)^{m_c-j+1}\p_zf+\sum_{j=0}^{m_c}C_{2,j}\Big(\f{\p_r}{r}\Big)^{j}\Big(\f{\p_\theta u_\theta}{r}+\p_zu_z\Big)\Big(\f{\p_r}{r}\Big)^{m_c-j+1}f\Bigg\}\\
&+\sum_{j=0}^{m_r-1}C_{m_r,j}\bigg(\partial_r^{j+1}\bl{u}\cdot \bar{\nabla} \partial_r^{m_r-1-j}\Big(\frac{\partial_r}{r}\Big)^{m_c}f\bigg)\,.\\
\end{split}
\ee
Here $C_{m_r,j}$ is a nonnegative integer and may change from line to line. Direct calculation shows that
\[
\begin{split}
\Big[\p_r^{m_r+1}\Big(\f{\p_r}{r}\Big)^{m_c}\,,\,\bl{u}\cdot\bar{\na}\Big]f=&\,\p_r\Big[\p_r^{m_r}\Big(\f{\p_r}{r}\Big)^{m_c}\,,\,\bl{u}\cdot\bar{\na}\Big]f+\p_r\Big(\bl{u}\cdot\bar{\na}\p_r^{m_r}\Big(\f{\p_r}{r}\Big)^{m_c}f\Big)\\
&-\bl{u}\cdot\bar{\na}\p_r^{m_r+1}\Big(\f{\p_r}{r}\Big)^{m_c}f\\
=&\,\p_r\Big[\p_r^{m_r}\Big(\f{\p_r}{r}\Big)^{m_c}\,,\,\bl{u}\cdot\bar{\na}\Big]f+\p_r\bl{u}\cdot\bar{\na}\p_r^{m_r}\Big(\f{\p_r}{r}\Big)^{m_c}f\,.
\end{split}
\]
Substituting the assumption \eqref{ASSk} in the above equation, one deduces
\be\label{ESE}
\begin{split}
&\Big[\p_r^{m_r+1}\Big(\f{\p_r}{r}\Big)^{m_c}\,,\,\bl{u}\cdot\bar{\na}\Big]f\\
=\,&\p_r^{m_r+1}\Bigg\{\sum_{j=1}^{m_c+1}C_{1,j}\Big(\f{\p_r}{r}\Big)^ju_z\Big(\f{\p_r}{r}\Big)^{m_c-j+1}\p_zf\Bigg.\\
&\hskip 2cm\Bigg.+\sum_{j=0}^{m_c}C_{2,j}\Big(\f{\p_r}{r}\Big)^{j}\Big(\f{\p_\theta u_\theta}{r}+\p_zu_z\Big)\Big(\f{\p_r}{r}\Big)^{m_c-j+1}f\Bigg\}\\
&+\sum_{j=0}^{m_r-1}C_{m_r,j}\p_r\bigg(\partial_r^{j+1}\bl{u}\cdot \bar{\nabla} \partial_r^{m_r-1-j}\Big(\frac{\partial_r}{r}\Big)^{m_c}f\bigg)+\p_r\bl{u}\cdot\bar{\na}\p_r^{m_r}\Big(\f{\p_r}{r}\Big)^{m_c}f\\
=&\p_r^{m_r+1}\Bigg\{\sum_{j=1}^{m_c+1}C_{1,j}\Big(\f{\p_r}{r}\Big)^ju_z\Big(\f{\p_r}{r}\Big)^{m_c-j+1}\p_zf\Bigg.\\
&\hskip 2cm\Bigg.+\sum_{j=0}^{m_c}C_{2,j}\Big(\f{\p_r}{r}\Big)^{j}\Big(\f{\p_\theta u_\theta}{r}+\p_zu_z\Big)\Big(\f{\p_r}{r}\Big)^{m_c-j+1}f\Bigg\}\\
&+\sum_{j=0}^{m_r}C_{m_r+1,j}\bigg(\partial_r^{j+1}\bl{u}\cdot \bar{\nabla} \partial_r^{m_r-1-j}\Big(\frac{\partial_r}{r}\Big)^{m_c}f\bigg)\,.\\
\end{split}
\ee
Recall the Leibniz formula, \eqref{ESE} indicates the validity of \eqref{MMM} for $\mathfrak{M}=(m_c,m_r,0)$, $\forall m_c,\,m_r\in\mathbb{N}$. Finally, an induction of $m_z$ indicates the validity of \eqref{MMM}. The proof is similar to the induction of $m_r$, thus we omit the details here.
\end{proof}

Given $3\leq m\in\mathbb{N}$ and $|\mathfrak{M}|=m$, acting $\mathcal{D}^\mathfrak{M}$ on \eqref{EHH}, one arrives at
\[
\p_t\mD^\mathfrak{M}\mH+\bl{v}\cdot\na\mD^\mathfrak{M}\mH-\mD^\mathfrak{M}\Big(\mH\f{\p_\theta v_\theta}{r}\Big)-2\mD^\mathfrak{M}(\mH\p_z\mH)=-[\mD^\mathfrak{M},\bl{v}\cd\bar{\na}]\mH.
\]
Multiplying the resulting equation with $\mD^\mathfrak{M}\mH$ and integrating over $\mR^3$, since $\mD^\mathfrak{M}\mH$ is independent with $\theta$, one has
\[
\int_{\bR^3}\bl{v}\cdot\na\mD^\mathfrak{M}\mH\cd\mD^\mathfrak{M}\mH dx=-\f{1}{2}\int_{\bR^3}\text{div}\,\bl{v}|\mD^\mathfrak{M}\mH|^2dx=0\,.
\]
Recall the aforementioned vanishing property \eqref{Van1}, we derive that
\[
\int_{\bR^3}\mD^\mathfrak{M}\mH\cdot\mD^\mathfrak{M}\Big(\mH\f{\p_\theta v_\theta}{r}\Big)dx=\int_{\bR^3}\f{\p}{\p\theta}\bigg(\mD^\mathfrak{M}\mH\cdot\mD^\mathfrak{M}\Big(\mH\f{v_\theta}{r}\Big)\bigg)dx=0\,.
\]
Thus one arrives at
\be\label{EH0}
\f{d}{dt}\|\mD^\mathfrak{M}\mH(t,\cd)\|_{L^2}^2\les \Big|\un{\int_{\mR^3}\mD^\mathfrak{M}(\mH\p_z\mH)\mD^\mathfrak{M}\mH dx}_{H_1}\Big|+\big|\un{\int_{\mR^3}[\mD^\mathfrak{M},\bl{v}\cdot\na]\mH\mD^\mathfrak{M}\mH dx}_{H_2}\Big|\,.
\ee
Here we split $H_1$ to
\[
H_1=\un{\int_{\mR^3}\mH\p_z\mD^\mathfrak{M}\mH\mD^\mathfrak{M}\mH dx}_{H_{11}}+\un{\int_{\mR^3}[\mD^\mathfrak{M},\mH\p_z]\mH\mD^\mathfrak{M}\mH dx}_{H_{12}}\,.
\]
Using integration by parts, $H_{11}$ follows that
\be\label{EH1}
|H_{11}|=\f{1}{2}\Big|\int_{\mR^3}\mH\p_z\big(\mD^\mathfrak{M}\mH\big)^2 dx\Big|=\f{1}{2}\Big|\int_{\mR^3}\p_z\mH|\mD^\mathfrak{M}\mH|^2dx\Big|\les\|\na\mH(t,\cd)\|_{L^\i}\|\mD^\mathfrak{M}\mH(t,\cd)\|_{L^2}^2\,.
\ee
By \eqref{MMM0} in Proposition \ref{PROP33} and the H\"older inequality, one deduces
\be\label{RCC}
\begin{split}
\|[\mD^\mathfrak{M},\mH\p_z]\mH(t,\cd)\|\les&\sum_{1\leq|\mathfrak{N}|\leq |\mathfrak{M}|}\|(\mathcal{D}^{\mathfrak{N}}\mH\mathcal{D}^{\mathfrak{M}-\mathfrak{N}}\p_z\mH)(t,\cd)\|_{L^2}\\
\les&\,\,\|\mH(t,\cd)\|_{H^m}\|\bar{\na}\mH(t,\cd)\|_{L^\i}\\
&+\sum_{2\leq|\mathfrak{N}|\leq |\mathfrak{M}|-1}\|\mathcal{D}^{\mathfrak{N}}\mH(t,\cd)\|_{L^4}\|\mathcal{D}^{\mathfrak{M}-\mathfrak{N}}\p_z\mH(t,\cd)\|_{L^4}\,.
\end{split}
\ee
Using the Sobolev imbedding, recalling $m=|\mathfrak{M}|\geq 3$, one notices
\[
\|\bar{\na}\mH(t,\cd)\|_{L^\i}\les\|\mH(t,\cd)\|_{H^m}\,.
\]
And since $2\leq|\mathfrak{N}|\leq m-1$, one has
\[
\begin{split}
\max\big\{\|\mathcal{D}^{\mathfrak{N}}\mH(t,\cd)\|_{L^4}\,,\,\|\mathcal{D}^{\mathfrak{M}-\mathfrak{N}}\p_z\mH(t,\cd)\|_{L^4}\big\}&\les\max\big\{\|\mathcal{D}^{\mathfrak{N}}\mH(t,\cd)\|_{H^1}\,,\,\|\mathcal{D}^{\mathfrak{M}-\mathfrak{N}}\p_z\mH(t,\cd)\|_{H^1}\big\}\\
&\les\|\mH(t,\cd)\|_{H^m}\,.
\end{split}
\]
Thus we arrive at
\[
\|[\mD^\mathfrak{M},\mH\p_z]\mH(t,\cd)\|_{L^2}\les\|\mH(t,\cd)\|_{H^m}^2\,,
\]
which indicates
\be\label{EH211}
|H_{12}|\les\|\mH(t,\cd)\|_{H^m}^3
\ee
by the Cauchy-Schwarz inequality. Combining \eqref{EH1} and \eqref{EH211}, one derives
\be\label{EH10}
|H_1|\les\|\mH(t,\cd)\|_{H^m}^3\,.
\ee

For $H_2$, using \eqref{MMM}, one derives
\[
\begin{split}
H_2=&\sum_{0\leq|\mathfrak{N}|\leq |\mathfrak{M}|-1}C_{3,\,\mathfrak{N}}\un{\int_{\mathbb{R}^3}\mathcal{D}^\mathfrak{N}\Big(\f{\p_\theta v_\theta}{r}\Big)\mathcal{D}^{\mathfrak{M}-\mathfrak{N}}\mH\mD^\mathfrak{M}\mH dx}_{H_{21}}\\
&+\sum_{1\leq|\mathfrak{N}|\leq |\mathfrak{M}|}C_{2,\,\mathfrak{N}}\un{\int_{\mathbb{R}^3}\mathcal{D}^{\mathfrak{N}}v_z\mathcal{D}^{\mathfrak{M}-\mathfrak{N}}\p_z\mH\mD^\mathfrak{M}\mH dx}_{H_{22}}\\
&+\sum_{0\leq|\mathfrak{N}|\leq |\mathfrak{M}|-1}C_{3,\,\mathfrak{N}}\un{\int_{\mathbb{R}^3}\mathcal{D}^\mathfrak{N}\p_zv_z\mathcal{D}^{\mathfrak{M}-\mathfrak{N}}\mH\mD^\mathfrak{M}\mH dx}_{H_{23}}\\
&+\sum_{1\leq|\mathfrak{L}|\leq |\bar{\mathfrak{M}}|}C_{4,\,\mathfrak{L}}\un{\int_{\mathbb{R}^3}\bar{\na}^{\mathfrak{L}}v_r\p_r\mathcal{D}^{\mathfrak{M}-(0,l_r,l_z)}\mH\mD^\mathfrak{M}\mH dx}_{H_{24}}\,.
\end{split}
\]
Here, since $\mD^\mathfrak{M}\mH$ is axially symmetric,
\be\label{VN}
\int_{\mathbb{R}^3}\mathcal{D}^\mathfrak{N}\Big(\f{\p_\theta v_\theta}{r}\Big)\mathcal{D}^{\mathfrak{M}-\mathfrak{N}}\mH\mD^\mathfrak{M}\mH dx=\int_{\mathbb{R}^3}\p_\theta\bigg\{\mathcal{D}^\mathfrak{N}\Big(\f{v_\theta}{r}\Big)\mathcal{D}^{\mathfrak{M}-\mathfrak{N}}\mH\mD^\mathfrak{M}\mH\bigg\} dx=0\,.
\ee
This means $H_{21}$ vanishes. By Corollary \ref{COR255}, one notices that for any multi-index $\mathfrak{P}$, there exists a scalar function $F_{\mathfrak{P}}(r,z,\theta)$, such that $\mD^\mathfrak{P}v_z+\p_\theta F_{\mathfrak{P}}(r,z,\theta)$ is a component of the gradient tensor $\na^{|\mathfrak{P}|}\bl{v}$. Akin to \eqref{VN}, one knows that the $\p_\theta F_{\mathfrak{P}}(r,z,\theta)$ part keeps silent when carrying out the $L^2$-based energy estimate. Also by Lemma \ref{LEMa}, $\bar{\na}^\mathfrak{L}v_r$ is a component of the gradient tensor $\na^{|\mathfrak{L}|}\bl{v}$. Therefore, recalling \eqref{RCC}-\eqref{EH211}, terms $H_{22}$-$H_{24}$ can be estimated in the following uniform way:
\be\label{EH3}
|H_2|\les\|\bl{v}(t,\cd)\|_{H^m}\|\mH(t,\cd)\|_{H^m}^2\,.
\ee

Combining estimates \eqref{EH10}, \eqref{EH3} of the right hand side of \eqref{EH0}, we deduce
\[
\f{d}{dt}\|\mD^\mathfrak{M}\mH(t,\cd)\|_{L^2}^2\les\|\mH(t,\cd)\|_{H^m}^3+\|\bl{v}(t,\cd)\|_{H^m}\|\mH(t,\cd)\|_{H^m}^2\,.
\]
Integrating with the temporal variable over $(0,t)$, and performing an interpolation with the fundamental energy bound in Lemma \ref{LEMFUND}, we arrive at the following estimate
\be\label{EMMH}
\|\mH(t,\cdot)\|^2_{{H}^m}\leq \|\mH_0\|^2_{{H}^m}+ C\int_0^t\|(\bl{v},\mH)(s,\cdot)\|^3_{{H}^m}ds.
\ee
Composing \eqref{ENVH} and \eqref{EMMH}, one concludes that
\[
\|(\bl{v},\bl{h},\mH)(t,\cdot)\|^2_{{H}^m}\leq \|(\bl{v}_0,\bl{h}_0,\mH_0)\|^2_{{H}^m}+ C\int_0^t\|(\bl{v},\bl{h},\mH)(s,\cdot)\|^3_{{H}^m}ds.
\]
This estimate implies a local-in-time bound for $\|(\bl{v},\bl{h},\mH)(t,\cdot)\|^2_{{H}^m}$, for any $m\geq3$ and $m\in\mathbb{N}$. That is, for some $T_*>0$ depending on $m$ and $\|(\bl{v}_0,\bl{h}_0,\mH_0)\|^2_{{H}^m}$ such that
\[
\|(\bl{v},\bl{h},\mH)(t,\cdot)\|^2_{{H}^m}\leq C_{0,T_*},\q\forall t\in[0,T_*],
\]
which obtains the local well-posedness of the system \eqref{(2.1)}.

\qed

\begin{remark}
There is an observation for the case of $m=1$. In fact, one may not able to find the trivialness of
\[
\int_{\bR^3}\mH\p_{x_i}\f{\p_\theta v_\theta}{r}\p_{x_i}\mH dx,\q\text{for } i=1,2,
\]
but their summation over $i=1,2$ do vanish. Here goes a short proof: Using the cylindrical coordinates and noticing $\mH$ is independent with $\theta$, direct calculation shows
\[
\begin{split}
&\sum_{i=1}^2\int_{\bR^3}\mH\p_{x_i}\f{\p_\theta v_\theta}{r}\p_{x_i}\mH dx=\f{1}{2}\sum_{i=1}^2\int_{\bR^3}\p_{x_i}\f{\p_\theta v_\theta}{r}\p_{x_i}\mH^2 dx=-\f{1}{2}\sum_{i=1}^2\int_{\bR^3}\f{\p_\theta v_\theta}{r}\p_{x_i}\p_{x_i}\mH^2 dx\\
=&-\f{1}{2}\int_{\bR^3}\f{\p_\theta v_\theta}{r}\Big(\p^2_{rr}+\f{1}{r}\p_r\Big)\mH^2dx=0\,.
\end{split}
\]
Here the last inequality follows from the observation \eqref{ENS} . However, this way seems far more complicated to go through when $m\geq2$. Thus it seems necessary, or at least convenient, to estimate higher order derivatives of the unknowns in the cylindrical coordinates, even if $\bl{v}$ is not axially symmetric.
\end{remark}

\qed

\section*{Acknowledgments}
\addcontentsline{toc}{section}{Acknowledgments}
\q\ The author wishes to thank the referees for many helpful comments. Part of this work was done when the author was visiting Department of Mathematics, Chung-Ang University, in the winter 2023. He is very grateful to Professor Dongho Chae for his support and helpful discussions on this topic during the visiting period. He also wishes to thank Professor Qi S. Zhang in UC Riverside for his valuable and profound comments on this work.

Z. Li is supported by the China Postdoctoral Science Foundation (No. 2024M763474) and National Natural Science Foundation of China (No. 12001285).

\medskip
\medskip

{\footnotesize

{\sc Z. Li: School of Mathematics and Statistics, Nanjing University of Information Science and Technology, Nanjing 210044, China, and Academy of Mathematics \& Systems Science, The Chinese Academy of Sciences,
Beijing 100190, China.}

  {\it E-mail address:}  zijinli@nuist.edu.cn

}
\end{document}